\documentclass[reqno, 12pt]{amsart}
\pdfoutput=1
\makeatletter
\let\origsection=\section \def\section{\@ifstar{\origsection*}{\mysection}} 
\def\mysection{\@startsection{section}{1}\z@{.7\linespacing\@plus\linespacing}{.5\linespacing}{\normalfont\scshape\centering\S}}
\makeatother        
\usepackage{amsmath,amssymb,amsthm}    
\usepackage{mathabx}\changenotsign   
\usepackage{mathrsfs} 
\usepackage{dsfont} 
\usepackage[babel]{microtype}
\usepackage{xcolor}  	
\usepackage[backref]{hyperref}
\hypersetup{
	colorlinks,
    linkcolor={red!60!black},
    citecolor={green!60!black},
    urlcolor={blue!60!black},
}

\usepackage{bookmark}

\usepackage[abbrev,msc-links,backrefs]{amsrefs} 
\usepackage{doi}

\renewcommand{\PrintDOI}[1]{\doi{#1}}

\usepackage[T1]{fontenc}
\usepackage{lmodern}

\usepackage[english]{babel}
\numberwithin{equation}{section}

\usepackage{geometry}
\usepackage{enumitem}
\def\rmlabel{\upshape({\itshape \roman*\,})}
\def\RMlabel{\upshape(\Roman*)}
\def\alabel{\upshape({\itshape \alph*\,})}

\def\aplabel{\upshape({\itshape \alph*\,$'$})}

\let\polishlcross=\l
\def\l{\ifmmode\ell\else\polishlcross\fi}

\def\qand{\quad\text{and}\quad}
\def\qqand{\qquad\text{and}\qquad}

\let\emptyset=\varnothing
\let\setminus=\smallsetminus

\let\subset\subseteq
\let\log=\ln

\newcommand\mc{\mathop{\textrm{\rm mc}}\nolimits}

\makeatletter
\def\moverlay{\mathpalette\mov@rlay}
\def\mov@rlay#1#2{\leavevmode\vtop{   \baselineskip\z@skip \lineskiplimit-\maxdimen
   \ialign{\hfil$\m@th#1##$\hfil\cr#2\crcr}}}
\newcommand{\charfusion}[3][\mathord]{
    #1{\ifx#1\mathop\vphantom{#2}\fi
        \mathpalette\mov@rlay{#2\cr#3}
      }
    \ifx#1\mathop\expandafter\displaylimits\fi}
\makeatother

\DeclareMathOperator{\dom}{{\rm dom}}

\newcommand{\dcup}{\charfusion[\mathbin]{\cup}{\cdot}}

\DeclareFontFamily{U}  {MnSymbolC}{}
\DeclareSymbolFont{MnSyC}         {U}  {MnSymbolC}{m}{n}
\DeclareFontShape{U}{MnSymbolC}{m}{n}{
    <-6>  MnSymbolC5
   <6-7>  MnSymbolC6
   <7-8>  MnSymbolC7
   <8-9>  MnSymbolC8
   <9-10> MnSymbolC9
  <10-12> MnSymbolC10
  <12->   MnSymbolC12}{}
\DeclareMathSymbol{\powerset}{\mathord}{MnSyC}{180}

\makeatletter
\def\namedlabel#1#2{\begingroup
    #2%
    \def\@currentlabel{#2}%
    \phantomsection\label{#1}\endgroup
}
\makeatother

\newtheorem{theorem}             {Theorem}[section]
\newtheorem{lemma}     	[theorem] {Lemma}

\newtheorem{definition}	[theorem] {Definition}   
\newtheorem{proposition}[theorem] {Proposition}   

\newtheorem{fact}	[theorem] {Fact}     
\newtheorem{claim}	[theorem] {Claim}  

\newtheoremstyle{remark}  {2pt}  {4pt}  {\rm}  {}  {\bfseries}  {.}  {.3em}          {}
\theoremstyle{remark}
\newtheorem{remark}	[theorem] {Remark}

\let\eps=\varepsilon
\let\theta=\vartheta
\let\rho=\varrho
\let\phi=\varphi

\def\NN{\mathds N}
 
\def\QQ{\mathds Q}
\def\RR{\mathds R}
\def\PP{\mathds P}
\def\EE{\mathds E}

\def\ra{\longrightarrow}
\usepackage{centernot}

\def\red{\text{\rm red}}
\def\blue{\text{\rm blue}}
\def\green{\text{\rm green}}
\def\R{\text{\rm red}}
\def\B{\text{\rm blue}}

\begin{document}

\title{Monochromatic trees in random graphs}

\author[Y. Kohayakawa]{Yoshiharu Kohayakawa}

\author[G. O. Mota]{Guilherme Oliveira Mota}
\address{Instituto de Matem\'atica e Estat\'{\i}stica, Universidade de
   S\~ao Paulo, S\~ao Paulo, Brazil}
\email{\{\,yoshi\,|\,mota\,\}@ime.usp.br}

\author[M. Schacht]{Mathias Schacht}
\address{Fachbereich Mathematik, Universit\"at Hamburg, Hamburg, Germany}
\email{schacht@math.uni-hamburg.de}

\thanks{%
  The first author was partially supported by FAPESP
  (Proc.~2013/03447-6, 2013/07699-0), by CNPq (Proc.~459335/2014-6,
  310974/2013-5) and by Project MaCLinC/USP.  The second author was
  supported by FAPESP (Proc.~2013/11431-2, 2013/20733-2) and partially
  by CNPq (Proc.~459335/2014-6).  The collaboration of the authors was
  supported by CAPES/DAAD PROBRAL (Proc. 430/15) and by FAPESP
  (Proc.~2013/03447-6).}

\begin{abstract}
  Bal and DeBiasio [\textit{Partitioning random graphs into monochromatic
  components}, Electron. J. Combin.~\textbf{24} (2017), Paper 1.18] put forward a
  conjecture concerning the threshold for the following Ramsey-type
  property for graphs~$G$: every $k$-colouring of the edge set of~$G$
  yields~$k$ pairwise vertex disjoint monochromatic trees that
  partition the whole vertex set of~$G$.  We determine the threshold
  for this property for two colours.
\end{abstract}


\maketitle

\section{Introduction}
For a graph $G=(V,E)$ we write $G\ra\Pi_2$ if for every 2-colouring of $E$, say with colours red and blue, 
there exist two monochromatic trees $T_1$, $T_2\subseteq G$ such that 
\[
	V(T_1)\dcup V(T_2)=V\,,
\] 
i.e., $V$ can be split into two sets each inducing a spanning
monochromatic component.  Here we allow one of the trees to be empty
and we also allow both trees to be monochromatic of the same colour.
In~\cite{BaDe}*{Conjecture~8.1} Bal and DeBiasio conjectured that if
\[
	p=p(n)>(1+\eps)\left(\frac{2\ln n}{n}\right)^{1/2}
\] 
for some $\eps>0$,
then \emph{asymptotically almost surely $($a.a.s.$)$} the binomial random graph~$G(n,p)$ satisfies $G(n,p)\ra\Pi_2$, i.e.,
\[
	\lim_{n\to\infty}\PP\big(G(n,p)\ra\Pi_2\big)=1\,.
\]
One can observe that the conjectured condition 
on~$p$ would be best possible. In fact, if $p<(1-\eps)\big(\frac{2\ln n}{n}\big)^{1/2}$ 
for some $\eps>0$, then a.a.s.\ $G(n,p)$ has diameter at least three (see, e.g.,~\cite{Bo01}*{Chapter~10})
and, hence, it
contains two non-adjacent vertices $u$ and $v$ with disjoint neighbourhoods.
Colouring all edges incident to $u$ or $v$ red and all other edges blue 
produces a colouring that requires at least three monochromatic trees
in any decomposition of~$V(G(n,p))$,
since $u$ and $v$ cannot be in the same red tree.

Bal and DeBiasio showed that a.a.s.~$G(n,p)\ra\Pi_2$ provided that
$p>C\big(\frac{\log n}{n}\big)^{1/3}$ for some suitable constant $C>1$.
We improve on that result by showing that~$\big(\frac{\log
  n}{n}\big)^{1/2}$ is the threshold for that property.

\begin{theorem}\label{thm:main}
	If $p=p(n)\gg\big(\frac{\log n}{n}\big)^{1/2}$, then a.a.s.\ $G(n,p)\ra\Pi_2$.
\end{theorem}

Combined with the discussion above, Theorem~\ref{thm:main} implies
that $\big(\frac{\log n}{n}\big)^{1/2}$ 
is the threshold for the property $G\ra\Pi_2$. We remark that our proof also yields a semi-sharp
threshold, since with not much additional effort we could replace the assumption $p\gg\big(\frac{\log n}{n}\big)^{1/2}$ by~$p>C\big(\frac{\log n}{n}\big)^{1/2}$ for some suitable constant $C>1$. 
However, for a simpler presentation 
we chose to avoid these calculations and we will only consider the case stated in Theorem~\ref{thm:main}.
In fact, since Theorem~\ref{thm:main} implies that the threshold function for the monotone graph property 
$G\ra\Pi_2$ is not of the form~$n^{-\alpha}$ for some rational $\alpha\in\QQ_{>0}$ it follows from Friedgut's 
criterion~\cite{Fr99}*{Theorem~1.4} that $G\ra\Pi_2$ has indeed a sharp threshold, i.e., there exist
constants $c_1>c_0>0$ and a function $c\colon \NN\to \RR$ with $c_0<c(n)<c_1$ for every $n\in\NN$
such that for every $\eps>0$ we have 
\[
	\lim_{n\to\infty}\PP\big(G(n,p)\ra\Pi_2\big)=\begin{cases}
		0,&\text{if $p<(1-\eps)c(n)\big(\frac{\log n}{n}\big)^{1/2}$}\\
		1,&\text{if $p>(1+\eps)c(n)\big(\frac{\log n}{n}\big)^{1/2}$}\,.
	\end{cases}
\]
In view of the question of Bal and DeBiasio~\cite{BaDe} it remains to
show that $c(n)$ is a constant independent of $n$ and that we have
$c(n)\equiv\sqrt2$.  

Finally, we remark that Bal and DeBiasio~\cite{BaDe} 
also considered multicoloured extensions of this problem and several other
interesting variants. 
Among other they proposed an extension of Theorem~\ref{thm:main} for $r$-colourings of the edges of~$G(n,p)$. 
More precisely, Bal and DeBiasio conjectured that if 
$p=p(n)>(1+\eps)\left(\frac{r\ln n}{n}\right)^{1/r}$ for some $\eps>0$, then a.a.s.\ every $r$-colouring 
of the edges of $G(n,p)$ admits a partition of $V(G)$ into at most $r$ sets each inducing 
a spanning monochromatic component.
It was noted by Ebsen, Mota, and Schnitzer~\cite{EMS} that this conjecture fails to be true 
and that for $r\geq 3$ the threshold for the partition property is at least 
$\big(\frac{\ln n}{n}\big)^{\frac{1}{r+1}}$. 
We present their example in Proposition~\ref{prop:counter} in Section~\ref{sec:counter}.

Roughly speaking, the proof of Theorem~\ref{thm:main}, given in
Section~\ref{sec:proof}, splits into two parts.  We shall define
what we mean by an \emph{extremal} colouring of the edges of a graph,
and we shall consider the extremal and the non-extremal cases
separately.  We shall first consider the somewhat simpler case of
non-extremal colourings in Section~\ref{sec:nonextremal}.  Extremal
colourings will be harder to handle and such colourings will be
analysed in Section~\ref{sec:extremal}.
Before the discussion of these two cases we collect a few observations 
concerning random graphs in Section~\ref{sec:gnp}.
 
\section{Preliminaries}
\label{sec:gnp}
We consider finite simple graphs and follow standard notation and
terminology (see~\cites{Bo98,BM08,Di10} and~\cites{Bo01,JLR00}).  We shall
make use of the following simple lemma on random graphs.

\begin{lemma}\label{lem:gnp}
  If~$p=p(n)\gg((\log n)/n)^{1/2}$, then  for every $\eps>0$ a.a.s. $G\in G(n,p)$ satisfies
  the following properties.
	\begin{enumerate}[label=\rmlabel]
	\item\label{it:gnp:1} Every vertex $v\in V(G)$ has degree $d_G(v)=(1\pm \eps)pn$ and
          every pair of distinct vertices $u$, $w\in V(G)$ 
          has $|N_G(u)\cap N_G(w)|=(1\pm\eps)p^2n$ joint neighbours.
	\item\label{it:gnp:2} For every vertex $v\in V(G)$ 
		and all disjoint subsets $U\subseteq V$
		and~$W\subseteq N_G(v)$ with $|U|\geq 100/p$ and $|W|\geq pn/100$ 
		the number~$e_G(U,W)$ of edges in the induced bipartite graph $G[U,W]$ satisfies $e_G(U,W)> p|U||W|/2$.
	\item\label{it:gnp:3} For every vertex $v\in V(G)$ and $J\subseteq N_G(v)$ with $|J|\geq pn/100$, we have that 
		all but at most $100/p$ vertices $x\in V(G)\setminus J$ satisfy $|N_G(x)\cap J|>p^2n/200$.
	\item\label{it:gnp:4} For every vertex $y\in V(G)$ and $A\dcup B=U\subseteq N_G(y)$ with $|U|\geq |N_G(y)|-p^2n/100$
		and $|A|$, $|B|\geq p^2n/2$, the induced bipartite graph $G[A,B]$ contains at least $p^2n/100$ 
		vertices of degree at least $p^2n/100$.
	\item\label{it:gnp:stree} Every subgraph $H\subseteq G$ with minimum degree 
		$\delta(H)\geq(1/2+\eps)pn$ is connected. 
	\item\label{it:gnp:dgen} Every subgraph~$H\subseteq G$ on at most $100/p$ vertices is $10\log n$-degenerate. 
	\end{enumerate}
\end{lemma}
\begin{proof} Properties~\ref{it:gnp:1}--\ref{it:gnp:dgen} in
  Lemma~\ref{lem:gnp} follow from the concentration of the binomial
  distribution.  In fact, property~\ref{it:gnp:1} is a direct
  consequence of Chernoff's inequality.  
  
  Property~\ref{it:gnp:2} also
  follows from  that inequality by the following argument. For disjoint 
	subsets $U$, $W\subseteq V$  Chernoff's inequality (see, e.g.,~\cite{JLR00}*{Theorem~2.1}) yields
	\[
		\PP\big(e_G(U,W)\leq \tfrac{1}{2}p|U||W|\big) \leq \exp(-p|U||W|/8)\,.
	\]
	Summing over all possible choices of $v\in V$ and all subsets
        $U\subseteq V$ and $W\subseteq N_G(v)$ considered in the
        property, we arrive at
	\begin{align*}
		\PP\big(\,\text{property~\ref{it:gnp:2} fails}\,\big)
		&\leq
		n\sum_{u\geq 100/p}\sum_{w\geq pn/100}\binom{n}{u}\binom{n}{w}p^w\exp(-puw/8)\\
		&\leq
		n\sum_{u\geq 100/p}\sum_{w\geq pn/100}\exp(u\ln n)\left(\frac{\textrm{e} n p}{w}\right)^w\exp(-puw/8)\\
		&\leq 
		n\sum_{u\geq 100/p}\sum_{w\geq pn/100}\exp(u\ln n+6w-puw/8)\,.
	\end{align*}
	Since $puw/16-6w\geq w/4$
        for $u\geq 100/p$ and, since
        $puw/16\geq u p^2n/1600 \gg u\ln n$ for~$w\geq pn/100$
        and~$p\gg((\log n)/n)^{1/2}$, it follows that
	\[
          \PP\big(\,\text{property~\ref{it:gnp:2} fails}\,\big)
          \leq n \sum_{100/p\leq u\leq n}\sum_{w\geq pn/100}\exp(-w/4)=o(1)\,,
	\]
	which concludes the proof of Lemma~\ref{lem:gnp}~\ref{it:gnp:2}.
		
	Property~\ref{it:gnp:3} follows from~\ref{it:gnp:2}. Given a vertex $v$ and a subset $J\subseteq N_G(v)$ of size 
	at least~$pn/100$ we consider the set 
	\[
		U=\big\{x\in V(G)\setminus J\colon |N_G(x)\cap J|\leq p^2n/200\big\}\,.
	\]
	Assuming for a contradiction that $|U|>100/p$ we infer from~\ref{it:gnp:2} that 
	\[
		e_G(U,J)>p|U||J|/2\geq p|U|\cdot pn/200=p^2n|U|/200\,,
	\]
	which contradicts the definition of the set $U$. Consequently, $|U|\leq 100/p$ and property~\ref{it:gnp:3}
	is established.

	The proof of property~\ref{it:gnp:4} makes use of the fact
        that a.a.s.\ for every $y\in V$ and every subset~$A\subseteq N_G(y)$
        with $p^2n/2\leq|A|\leq|N_G(y)\setminus A|$ we have
	\begin{equation}\label{eq:smA2}
		e_G\big(A,N_G(y)\setminus A\big) 
		> 
		\frac{4}{25}p^2 n|A|\,.
	\end{equation}
	In fact, property~\ref{it:gnp:4} follows from~\eqref{eq:smA2} and we prove this implication first. 
        Let a vertex~$y$ and sets~$A$, $B$ and~$U$ be as in the statement of~\ref{it:gnp:4}.
        Without loss of generality, we may suppose~$|A|\leq|B|\leq|N_G(y)\setminus A|$,
        and hence we can apply~\eqref{eq:smA2}.
	Removing all vertices from $A$ that have less than $p^2n/50$ neighbours in $N_G(y)\setminus A$
	and using the bound $|N_G(y)\cap N_G(a)|\leq 2p^2n$ for all $a\in A$, which is given by~\ref{it:gnp:1}, 
	we deduce from~\eqref{eq:smA2} that at least 
	\[
		\frac{4p^2 n|A|/25-|A|p^2n/50}{2p^2n}=\frac{7|A|}{100}>\frac{p^2n}{100}
	\]
	vertices in~$A$ have at least $p^2n/50$ neighbours in $N_G(y)\setminus A$.
	Since $B=(N_G(y)\setminus A)\setminus B'$ for some $|B'|\leq
        p^2n/100$, property~\ref{it:gnp:4} follows and it is left to verify~\eqref{eq:smA2}.
	
	For the proof of~\eqref{eq:smA2} we may assume that $|A|\leq |N_G(y)\setminus A|$ 
	and we consider two cases depending on the size of~$A$. If $|A|\geq 100/p$ 
	inequality~\eqref{eq:smA2} is a consequence of
        property~\ref{it:gnp:2} applied with~$v=y$ and the disjoint
        sets $A$ and $N_G(y)\setminus A$ combined with the first part
        of~\ref{it:gnp:1}, which 
	leads to
	\[
		e_G\big(A,N_G(y)\setminus A\big) 
		\overset{\text{\ref{it:gnp:2}}}{\geq} 
		\frac{1}{2}p\big|A\big|\big|N_G(y)\setminus A\big|
		\overset{\text{\ref{it:gnp:1}}}{\geq}
		\frac{1}{2}p\big|A\big|\cdot\frac{1}{3}pn
		>
		\frac{4}{25}p^2 n|A|\,.
	\]
	For the case $|A|\leq 100/p$ we have $p^2n|A| \gg p|A|^2$.
        Hence, we may use the concentration inequality
        $\PP(X>t)\leq \exp(-t)$ for binomially distributed random
        variables $X$ satisfying~$\EE[X]\leq t/7$ (see,
        e.g.,~\cite{JLR00}*{Corollary~2.4}) to derive that, for every
        fixed set~$A$, we have
	\begin{equation*}
		\PP\big(2e_G(A) > p^2n|A|/4\big)
		\leq
		\exp(-p^2n|A|/4)\,.
	\end{equation*}
	Summing over all sets $A$ of size at most $100/p$ yields 
        \begin{multline}
          \label{eq:smA}
            \qquad\PP\big(\,\exists A\subset V\text{ with }|A|\leq100/p
            \text{ such that }2e_G(A) > p^2n|A|/4\big)\\
            \leq
            \sum_{a=p^2n/4}^{100/p}n^a\exp(-p^2n a/4)=o(1)\,,\qquad
        \end{multline}
	where the last inequality follows from our assumption on~$p$. We infer~\eqref{eq:smA2} from~\eqref{eq:smA}. 
	Given $y\in V(G)$ and $A\subseteq N_G(y)$ with $p^2n/2\leq |A|\leq 100/p$
	we appeal to the second assertion of property~\ref{it:gnp:1} with $\eps=1/2$
	for all pairs of the form $y$, $a$ with $a\in A$. Summing $|N_G(y)\cap N_G(a)|$ over all $a\in A$ 
	yields
 	\begin{equation*}
 		e_G\big(A,N_G(y)\setminus A\big) 
 		> 
 		\frac{1}{2}p^2 n|A|-2e_G(A)
		\overset{\eqref{eq:smA}}{>}
 		\frac{1}{6}p^2 n|A|
 	\end{equation*}
	and~\eqref{eq:smA2} follows. This concludes the proof of~property~\ref{it:gnp:4}.
		
	For property~\ref{it:gnp:stree} we observe that for
        $p\gg(\log n)/n$ and every fixed $\delta>0$, again
        Chernoff's inequality implies that a.a.s., for every subset
        $U\subseteq V$, we have
	\begin{equation}\label{eq:connect}
		2e_G(U) < p|U|^2+\delta pn|U|\,.
	\end{equation}
        To prove~\eqref{eq:connect}, one can analyse the cases in
        which~$\delta n/|U|\leq3/2$, $3/2<\delta n/|U|<7$
        and~$\delta n/|U|\geq7$ separately.  For the first two cases,
        one can use one of the standard forms of Chernoff's
        inequality, as given in, e.g.,~\cite{JLR00}*{Corollary~2.3}.
        For the third case, one can again
        use~\cite{JLR00}*{Corollary~2.4}.  

        Next we consider an arbitrary component $C$ of the subgraph
        $H\subseteq G$ and let $U=V(C)$.  Combining~\eqref{eq:connect}
        for $\delta=\eps$ with the minimum degree assumption tells us
        that 
	\[
		|U|\cdot (1/2+\eps)pn \leq 2e_G(U) < p |U|^2 + \eps pn |U|\,,
	\]
	which implies $|U|>n/2$. Consequently, every component of $H$ spans more than $n/2$ vertices, which implies
	that $H$ is connected.
	
	For the proof of~\ref{it:gnp:dgen} it suffices to show that every subset $U\subseteq V$ of size at most $100/p$
	contains a vertex of degree at most $10\log n$. However, this follows from the observation that for every such set $U$
	we have 
	\[
		e_G(U)\leq |U|\cdot 5\log n\,,
	\]  
	which again can be deduced from the concentration inequality
        given in~\cite{JLR00}*{Corollary~2.4}.
\end{proof}

\section{Proof of the main result}
\label{sec:proof}
We introduce some further notation and classify the two-colourings into two classes 
(see Definition~\ref{def:extremal} below).
For a colouring $\phi \colon E\to\{\red,\blue\}$ of the edges of a graph~$G=(V,E)$ we write 
$\phi\ra\Pi_2$ to indicate that there exist two monochromatic trees $T_1$, $T_2\subseteq G$ such that 
\[
	V(T_1)\dcup V(T_2)=V\,.
\] 
In particular, $G\ra\Pi_2$ if $\phi\ra\Pi_2$ holds for all $2$-colourings $\phi$ of $E$. 
We denote the two edge disjoint spanning monochromatic subgraphs induced by 
$\phi$ by $G^\phi_\R$ and $G^\phi_\B$, i.e.,
\[
	G^\phi_\R=\big(V,\phi^{-1}(\red)\big)
	\qand
	G^\phi_\B=\big(V,\phi^{-1}(\blue)\big)\,.
\]
For a vertex $v\in V$ we consider its \emph{$\red$-} and \emph{$\blue$-neighbourhood}
\[
	N^\phi_\R(v)=\{u\in N(v)\colon \phi(\{v,u\})=\red\}
	\qand
	N^\phi_\B(v)=\{u\in N(v)\colon \phi(\{v,u\})=\blue\}
\]
and the corresponding degrees $d^\phi_\R(v)=|N^\phi_\R(v)|$ and $d^\phi_\B(v)=|N^\phi_\B(v)|$. We roughly 
classify the vertices depending on these degrees by defining the following sets
\begin{equation}\label{eq:BR}
	R^\phi=\big\{v\in V\colon d^\phi_\R(v)>\tfrac{1}{3}d(v)\big\}
	\qand
	B^\phi=\big\{v\in V\colon d^\phi_\B(v)>\tfrac{1}{3}d(v)\big\}\,.
\end{equation}
These sets might not be disjoint, but every vertex is a member of at least one of them
and vertices $v$ in the symmetric difference of these sets have at least $2d(v)/3$ 
neighbours in one colour.
In the proof of Theorem~\ref{thm:main} we consider two cases depending, whether there is 
a monochromatic path between some vertex in~$R^\phi$ and a different vertex in~$B^\phi$.
\begin{definition}\label{def:extremal}
	Let $G=(V,E)$ be a graph and $\phi\colon E\to\{\red,\blue\}$.
	We say $\phi$ is \emph{extremal} if there is a pair of distinct vertices 
	$r\in R^\phi$ and $b\in B^\phi$ for which no monochromatic $r$-$b$-path exists.
	If no such pair of vertices exists, then we say $\phi$ is \emph{non-extremal}. 
\end{definition}
For the proof of Theorem~\ref{thm:main} we consider non-extremal and
extremal colourings $\phi$ separately.  Before we proceed, let us
remark that the property $G\ra\Pi_2$ is an increasing property, that
is, if~$G$ is a spanning subgraph of~$G'$ and $G\ra\Pi_2$ holds,
then~$G'\ra\Pi_2$ also holds.  This implies that it suffices to prove
Theorem~\ref{thm:main} under the additional hypothesis that~$p=o(1)$.

\subsection{Non-extremal colourings}
\label{sec:nonextremal}

The following proposition addresses the case when $\phi$ is non-extremal. 

\begin{proposition}[Non-extremal case]
  \label{prop:nex}
  If $p=p(n)\gg((\log n)/n)^{1/2}$ and~$p=o(1)$, then
  a.a.s.~$G\in G(n,p)$ satisfies $\phi\to\Pi_2$ for every
  non-extremal colouring $\phi\colon E(G)\to\{\red,\blue\}$.
\end{proposition}

In the proof of Proposition~\ref{prop:nex} we shall make use of the following simple observation, which 
is closely related to the fact  that every $2$-colouring of the edges of the 
complete graph yields a monochromatic spanning tree.
\begin{lemma}
\label{lem:neat}
	Let $G=(V,E)$ be a graph and $\phi\colon E\to\{\red,\blue\}$. If for a subset $U\subseteq V$ all
	pairs of vertices $u$, $u'\in U$ are connected by a monochromatic path, then there exists a monochromatic 
	tree~$T$ with $V(T)\supseteq U$. 
\end{lemma}
\begin{proof}
	Let $T$ be a monochromatic tree containing the maximum number of vertices from~$U$.
	We may assume that $T$ is colored $\red$.
	If there is some vertex $u\in U\setminus V(T)$, then it must be connected to every vertex 
	$u'\in U\cap V(T)$ by a $\blue$ $u$-$u'$-path, which results in a monochromatic tree containing 
	at least one more vertex from $U$ than~$T$.
\end{proof}

With this observation at hand we can now establish the proof of the proposition.

\begin{proof}[Proof of Proposition~\ref{prop:nex}]
  Owing to $p\gg\big(\frac{\log n}{n}\big)^{1/2}$ we may and shall
  assume that for $\eps=1/10$ the graph $G=(V,E)\in G(n,p)$ satisfies
  properties~\ref{it:gnp:1}--\ref{it:gnp:dgen} given in
  Lemma~\ref{lem:gnp}. 
        Moreover, let 
	$\phi\colon E\to\{\red,\blue\}$ be a non-extremal colouring, which is fixed throughout 
	the proof. For simpler notation, we suppress the superscript~$\phi$ in terms like 
	$G^{\phi}_\R$, $N_\R^{\phi}(v)$, $d_\R^{\phi}(v)$, $R^\phi$, and their $\blue$ counterparts.
	
	If one of the sets $R$ or $B$, say~$R$, is empty, then it
        follows from property~\ref{it:gnp:1}
	that every vertex in~$G$ satisfies $d_\B(v)\geq (2/3-\eps)pn$. Hence, by 
	property~\ref{it:gnp:stree} 
	there exists a $\blue$ spanning tree of~$G$ and $\phi\ra\Pi_2$.
	
	Since~$\phi$ is non-extremal, between every vertex $r\in R$ and every $b\in B$ there exists 
	a monochromatic $r$-$b$-path. In particular, vertices contained in the intersection $R\cap B$
	are connected to every other vertex by a monochromatic path.
	
	 Below we show that there exist monochromatic components $C_\R\subseteq G_\R$ and
	 $C_\B\subseteq G_\B$ covering $V$, i.e., 
	 \begin{equation}\label{eq:CBR}
		V(C_\B)\cup V(C_\R)=V\,.
	 \end{equation}
	 Consider a  monochromatic component~$C$ containing the most number of vertices. In particular,
	 any pair of vertices in~$C$ can be connected by a monochromatic path. If $C$ would be 
	 completely contained in $R$ or $B$, say without loss of generality in $R$, then we can fix an 
	 arbitrary vertex $v\in B$ and Lemma~\ref{lem:neat}
	 would show that there exists a monochromatic component containing~$C$ and~$v$, which violates the maximal
	 choice of~$C$. Therefore, $C$ intersects each set $R$ and $B$ in at least one vertex, 
	 say $v_r\in R$ and $v_b\in B$ 
	 and without loss of generality we may assume $C$ is coloured~$\red$. 
	 
	 Then for every vertex 
	 $u\in R\setminus V(C)$ the monochromatic $v_b$-$u$-path must be $\blue$ and, hence, all pairs of vertices
	 in $R\setminus V(C)$ are connected by a $\blue$ path. Consequently, all pairs of vertices in
	 \begin{equation}\label{eq:easy1}
	 	\big(V(C)\cap B\big)\cup \big(R\setminus V(C)\big)
	 \end{equation}
	 are connected by monochromatic paths and another application of Lemma~\ref{lem:neat} yields 
	 a monochromatic component $C'$ containing the vertices from~\eqref{eq:easy1}. Similarly, there 
	 exists a monochromatic component $C''$ containing all vertices from
	 \[
	 	\big(V(C)\cap R\big)\cup \big(B\setminus V(C)\big)\,.
	 \]
	In particular, $C'$ and $C''$ cover all vertices of~$G$. If both these components have the same colour
	then we either found two disjoint monochromatic trees covering~$V$ or one such tree, i.e., $\phi\ra\Pi_2$. 
	If $C'$ and $C''$ are of different colours then~\eqref{eq:CBR} follows.
	
	It is left to deduce the proposition from~\eqref{eq:CBR}. Let $C_\R\subseteq G_\R$ and
	 $C_\B\subseteq G_\B$ satisfy~\eqref{eq:CBR}. We may assume that both components are maximal, i.e., 
	 every vertex in the complement of $C_\R$ has only $\blue$ neighbours in $C_\R$ and, analogously, 
	 every vertex in the complement of $C_\B$ has only $\red$ neighbours in $C_\B$. We consider 
	 the symmetric difference of $C_\R$ and $C_\B$ and let
	 \[
	 	O_\R=V(C_\R)\setminus V(C_\B)
		\qand
		O_\B=V(C_\B)\setminus V(C_\R)
	 \]
	 be the two parts of the symmetric difference, where vertices in $O_\R$ are 
	 only contained in $C_\R$ and those from $O_\B$ are only contained in~$C_\B$. 
	 Note that the maximal choice of 
	 $C_\R$ and $C_\B$ implies that there is no edge between $O_\R$ and $O_\B$. In fact, 
	 there is not even a monochromatic 
	 path between $O_\R$ and $O_\B$, since every edge leaving $O_\R$
	 is $\blue$ and every edge entering $O_\B$ is $\red$. Owing to the assumption 
	 that every vertex in~$R$ is connected by a monochromatic path with every vertex 
	 in~$B$ we arrive at one of the following two cases
	 \begin{enumerate}[label=\RMlabel]
	 	\item\label{it:nex:c1} $O_\R=\emptyset$ or $O_\B=\emptyset$,
	 	\item\label{it:nex:c2} $O_\R\cup O_\B\subseteq R\setminus B$ 
			or  $O_\R\cup O_\B\subseteq B\setminus R$.
	 \end{enumerate}
	 To see that one of the cases must occur, let us assume case~\ref{it:nex:c1}
	 does not hold and let $v\in O_\R$ and $u\in O_\B$.
	 As noted above it is not possible that one of the vertices is contained 
	 in~$R$, while the other one is a member of~$B$. Consequently, both of them 
	 must be contained in~$R\setminus B$ or in~$B\setminus R$. Repeating the same 
	 argument for pairs $(v,u')$ with $u'\in O_\B$ and pairs~$(v',u)$ with $v'\in O_\R$
	 yields case~\ref{it:nex:c2}.

	 Next we note that case~\ref{it:nex:c1} asserts that one of the parts of the symmetric difference of 
	 $C_\R$ and $C_\B$ is empty, which combined with~\eqref{eq:CBR} implies the existence 
	 of a monochromatic spanning tree in~$G$.
	 
	 For case~\ref{it:nex:c2} we can assume without loss of generality that 
	 $O_\R\cup O_\B\subseteq R\setminus B$.
	 We infer from the maximality of $C_\R$ that no vertex in~$O_\B$ has a $\red$ 
	 neighbour in $C_\R$, and, therefore,
	 \[
	 	N_\R(v)\subseteq O_\B
	 \] 
	 for every $v\in O_\B$. Since $O_\B\subseteq R\setminus B$ it follows from 
	 property~\ref{it:gnp:1} that $G_\R$ induced on~$O_\B$ 
	 has minimum degree $(2/3-\eps)pn$. Consequently,
         property~\ref{it:gnp:stree}
	 yields a $\red$ spanning tree on $O_\B$ and combined with a $\red$ spanning tree on $C_\R$
	 we found two vertex disjoint red trees covering $G$, which concludes the proof of 
	 Proposition~\ref{prop:nex}.
\end{proof}

\subsection{Extremal colourings}
\label{sec:extremal}
In this section we consider extremal colourings $\phi$ and establish an analogous proposition 
as in the non-extremal case. Together Propositions~\ref{prop:nex} and~\ref{prop:ex} establish 
Theorem~\ref{thm:main}.

\begin{proposition}[Extremal case]
  \label{prop:ex}
  If $p=p(n)\gg\big((\log n)/n\big)^{1/2}$ and~$p=o(1)$, then
  a.a.s.~$G\in G(n,p)$ satisfies $\phi\ra\Pi_2$ for every extremal
  colouring $\phi\colon E(G)\to\{\red,\blue\}$.
\end{proposition}
\begin{proof}
	As in the proof of Proposition~\ref{prop:nex} we may and shall
        assume that $G=(V,E)\in G(n,p)$
	satisfies properties~\ref{it:gnp:1}--\ref{it:gnp:dgen} for $\eps=1/100$ given in Lemma~\ref{lem:gnp}. Let 
	$\phi\colon E\to\{\red,\blue\}$ be a fixed extremal colouring and again, 
	for simpler notation, in what follows we suppress the
        superscript~$\phi$ in terms like~$G^{\phi}_\R$,  
	$N_\R^{\phi}(v)$, $d_\R^{\phi}(v)$, $R^\phi$, and their $\blue$ counterparts.
	
	Let $r\in R$ and $b\in B$ be two distinct vertices for which no monochromatic $r$-$b$-path 
	exists. We shall build a $\red$ and a $\blue$ tree with \textit{roots}~$r$ and~$b$. 
        We sometimes refer to~$r$ as the \textit{red root} and to~$b$ as the \textit{blue root}.  The trees will be  
	built in two stages. In the first stage every vertex $v\in V\setminus\{r,b\}$ will be assigned 
	a \emph{preferred colour} $\rho(v)$, which indicates its
        ``preference''. In fact, the preferred colour~$\rho(v)$ will 
	be chosen in such a way that $v$ can be connected in the `right colour' to~$r$ or~$b$ in a 
	robust way, that is, there will be `many' $\rho(v)$-coloured
        paths from~$v$ to the root of colour~$\rho(v)$.  The preferred
        colours will be assigned vertex by vertex and earlier choices
        may influence  
	those chosen later. However, in this process it might turn out that a later vertex~$v$ 
	needs to be connected to the $\blue$ tree through an earlier
        vertex $u$ with $\rho(u)=\red$ (thus~$u$ would in principle
        belong to the red tree that we are building).  To resolve such 
        conflicts, we finalise the choices in a second round after every vertex has chosen 
	its preferred colour and, in fact, here some vertices may get connected to the tree opposite
	to its preferred colour (e.g., because of~$v$ above we may
        decide to override~$u$'s preference ($\rho(u)=\red$) and
        connect~$u$ to the blue tree). 	
	Below we give the details of this approach.
	
	\subsection*{First stage: choosing preferred colours}
	We begin with the neighbours of~$r$ and~$b$ which are connected by an edge of the `right colour' to the 
	respective root.  For those vertices~$v$, we set the preferred colour to the obvious choice:
	\begin{equation}\label{eq:prefN}
		\rho(v)=
		\begin{cases}
			\red, &\text{if $v\in N_\R(r)\setminus N_\B(b)$}\\
			\blue, &\text{if $v\in N_\B(b)\setminus N_\R(r)$}\,.
		\end{cases}
	\end{equation}
	For symmetry reasons we defer the assignment of~$\rho(v)$ to
        the vertices~$v$ in $N_\R(r)\cap N_\B(b)$ for a moment.  Next
        we consider the edges between $N_\R(r)$ and $N_\B(b)$.  Recall
        that we assume that
        properties~\ref{it:gnp:1}--\ref{it:gnp:dgen} in
        Lemma~\ref{lem:gnp} hold for~$G$.  Recall also that we suppose
        that~$p=o(1)$.  Both assertions in property~\ref{it:gnp:1},
        combined with the definition of the sets $R$ and $B$, allow us
        to invoke property~\ref{it:gnp:2} to obtain that
	\[
		e_G\big(N_\R(r)\setminus N_\B(b), N_\B(b)\setminus N_\R(r)\big)
		\geq 
		\frac{p}{2}\,\big|N_\R(r)\setminus N_\B(b)\big|\,\big|N_\B(b)\setminus N_\R(r)\big|\,.
	\]
	At least half of these edges have the same colour and, by
        symmetry, \textit{we may assume that they are red}.
        We continue with the following claim.

        \begin{claim}
          \label{claim:1}
          At least~$pn/100$ vertices~$v\in N_\B(b)\setminus N_\R(r)$ 
          satisfy
          \begin{equation}\label{eq:J0}
            \big|N_\R(v)\cap \big(N_\R(r)\setminus N_\B(b)\big)\big|
            >
            \frac{p^2n}{25}\,.
          \end{equation}
        \end{claim}
        \begin{proof}
          The vertices~$v\in N_\B(b)\setminus N_\R(r)$ with
          \begin{equation}
            \label{eq:1}
            |N_\R(v)\cap(N_\R(r)\setminus
            N_\B(b))|\leq\frac{p}{8}|N_\R(r)\setminus N_\B(b)|
          \end{equation}
          can account for at most
          $(p/8)|N_\R(r)\setminus N_\B(b)||N_\B(b)\setminus
          N_\R(r)|$
          red edges between the sets $N_\R(r)\setminus N_\B(b)$
          and~$N_\B(b)\setminus N_\R(r)$, of which there are at
          least
          \[
          	\frac{1}{4}p\big|N_\R(r)\setminus N_\B(b)\big|\big|N_\B(b)\setminus N_\R(r)\big|\,.
          \]
          Therefore, in view of property~\ref{it:gnp:1}, there must
          be at least
          \begin{equation}\label{eq:Js}
            \frac{\frac{p}{8}\,\big|N_\R(r)\setminus
              N_\B(b)\big|\,\big|N_\B(b)\setminus
              N_\R(r)\big|}{(1+\eps)p^2n} 
            >
            \frac{1}{25}\big|N_\B(b)\setminus N_\R(r)\big|
            >
            \frac{pn}{100}
          \end{equation}
          vertices~$v\in N_\B(b)\setminus N_\R(r)$ with
          \begin{equation}\label{eq:J0b}
            \big|N_\R(v)\cap \big(N_\R(r)\setminus N_\B(b)\big)\big|
            >
            \frac{p}{8}\big|N_\R(r)\setminus N_\B(b)\big|
            >
            \frac{p^2n}{25}\,,
          \end{equation}
          as required. 
        \end{proof}

	The vertices~$v$ satisfying~\eqref{eq:J0} play a special r\^ole in the proof, since they can be  
	used to connect other vertices to both roots, as they are blue neighbours of~$b$ and connect (robustly)
	by $\red$ paths of length two to~$r$. Furthermore, the vertices in $N_\R(r)\cap N_\B(b)$ 
	are even direct neighbours of both roots in the right colour. We will refer to the vertices in
        \begin{equation}
          \label{eq:J_def}
		J
		=
		\big\{v\in N_\B(b)\setminus N_\R(r)\colon \text{$v$ satisfies \eqref{eq:J0}}\big\} 
		\cup 
		\big(N_\R(r)\cap N_\B(b)\big)
        \end{equation}
	as the \emph{joker vertices}.  Note that Claim~\ref{claim:1}
        implies
	\begin{equation}\label{eq:Jsize}
		|J| > \frac{pn}{100}\,.
	\end{equation}
	For the presentation, it will also be simpler to give all joker vertices the 
	same preferred colour and, hence, we set
	\[
		\rho(v)=\blue
	\]
	for all $v\in N_\R(r)\cap N_\B(b)$.  This way we have defined
        $\rho(v)$ for every~$v\in N_\R(r)\cup N_\B(b)$.
	
	Among the vertices not considered so far we turn first to those with 
	a decent number of joker vertices as neighbours. More precisely, we
	set
        \begin{equation}
          \label{eq:X_def}
          X=\Big\{x\in V\setminus \big(N_\R(r)\cup N_\B(b)\cup\{r,b\}\big)\colon
			\big|N(x)\cap J\big|>\frac{p^2n}{200}\Big\}\,.
        \end{equation}
	In particular, every vertex $x\in X$ has more than $p^2n/400$ jokers as neighbours in one colour 
	and this will 
	be its preferred colour, i.e., for every $x\in X$ we set
	\begin{equation}\label{eq:prefX}
		\rho(x)=\begin{cases}
			\red, & \text{if}\ |N_\R(x)\cap J|>\frac{p^2n}{400}\\
			\blue, & \text{if}\ |N_\B(x)\cap J|>\frac{p^2n}{400},
		\end{cases}
	\end{equation}
        for vertices~$x$ satisfying both conditions in~\eqref{eq:prefX},
        we pick the value of~$\rho(x)$ arbitrarily.
	Note that, for every vertex $v$ which has been assigned a
        preferred colour~$\rho(v)$ already,  
	\begin{equation}\label{eq:inv1}
          \text{there exists a $\rho(v)$-coloured  
            path from~$v$ to the root of colour~$\rho(v)$.}
	\end{equation}
	We shall keep this invariant in the assignment of the preferred colours to the remaining vertices.
	
        Before we continue, we make the following remark, which partly
        explains some of the underlying ideas in our approach.

        \begin{remark}\label{rem:1}
          If we have reached every vertex of~$G$ at this point (that
          is, if $V=\{r,b\}\cup N_\R(r)\cup N_\B(b)\cup X$), then we
          can finish the proof as follows. For every vertex in $J$ we
          decide independently with probability $1/2$ whether we
          attach it to the $\red$ tree or to the $\blue$ tree and
          every other vertex will be attached to the tree matching its
          preferred colour. This clearly works for the vertices in
          $N_\R(r)\cup N_\B(b)$. Moreover, since every vertex $x\in X$
          connects to at least $\frac{p^2n}{400}\gg \log n$ neighbours
          in~$J$ in its preferred colour, at least one of those
          neighbours will obtain that colour in the random assignment
          (with high probability) and this would conclude the
          proof. Note that, for this argument to work, it would
          suffice if the joker vertices
          in~$N_\B(b)\setminus N_\R(r)$ had just one red neighbour
          in $N_\R(r)\setminus N_\B(b)$.\hfill$\sqbullet$
        \end{remark}

        Unfortunately, some vertices may have only a few neighbours
        in~$J$,
        and therefore we could have that $V\neq\{r,b\}\cup
        N_\R(r)\cup N_\B(b)\cup X$.  Let
	\[
		Y=V\setminus \big(N_\R(r)\cup N_\B(b)\cup\{r,b\}\cup X\big)\,.
	\]
        We now proceed to define~$\rho(y)$ for every~$y\in Y$.  Since
        $J\subseteq N_\B(b)$ we can apply property~\ref{it:gnp:3} to
        obtain that
	\begin{equation}\label{eq:Ysize}
		m=|Y|\leq \frac{100}{p}\,.
	\end{equation}
	Consequently, we infer from property~\ref{it:gnp:dgen} that we can 
	order the vertices in~$Y$ as~$y_1,\dots,y_m$ in such a way that for every $i\in [m]$ we have 
	\begin{equation}\label{eq:Ny}
		\big|N(y_i)\cap Y_{i+1}\big|\leq 10\log n\quad
		\text{for $Y_{i+1}=\{y_{i+1},\dots,y_m\}$.}
	\end{equation}
	We shall assign the preferred colours to the vertices in~$Y$ in this order. Let 
	$i\in [m]$ and suppose the preferred colours $\rho(y_j)$ for $j\in[i-1]$ were already 
	fixed. We consider two cases depending on the preferred colours appearing in the 
	neighbourhood of~$y_i$. We split~$N(y_i)$ according to the preferred colours of the vertices, i.e., 
	we consider the partition
	\[
		N(y_i)=\big(N(y_i)\cap\rho^{-1}(\red)\big)
		\dcup
		\big(N(y_i)\cap\rho^{-1}(\blue)\big)
		\dcup
		\big(N(y_i)\cap Y_{i+1}\big)\,.		
	\]

	We say $y_i$ is \emph{canonically connected in $\red$ $($resp.\ $\blue$$)$} if~$y_i$ connects in 
	$\red$ (resp.\ $\blue$) to many vertices with preferred colour $\red$ (resp.\ $\blue$), i.e.,
	\begin{equation}\label{eq:prefYs}
		\big|N_\R(y_i)\cap\rho^{-1}(\red)\big|\geq \frac{p^2n}{400}
	\end{equation}
	(resp.\ $|N_\B(y_i)\cap\rho^{-1}(\blue)|\geq p^2n/400$). 
	If $y_i$ fails to be canonically connected in 
	either colour, then we say it is \emph{non-canonically connected}.
	
	We set $\rho(y_i)=\red$ (resp.\ $\rho(y_i)=\blue$) if $y_i$ is canonically connected in red (resp.\ blue).
	Clearly, by induction, 
	with this choice of $\rho(y_i)$ we also ensure property~\eqref{eq:inv1}.
	
	It is left to consider vertices $y_i$ that are non-canonically
        connected.  Since
	\begin{multline*}
		\big(N_\B(y_i)\cap\rho^{-1}(\red)\big)
		\dcup
		\big(N_\R(y_i)\cap\rho^{-1}(\blue)\big)\\
		=
		N(y_i)\setminus
			\Big(\big(N_\R(y_i)\cap\rho^{-1}(\red)\big)
			\dcup
			\big(N_\B(y_i)\cap\rho^{-1}(\blue)\big)
			\dcup
			\big(N(y_i)\cap Y_{i+1}\big)
			\Big)\,,
	\end{multline*}
	in this case
	we have 
	\begin{align}
		\big|\big(N_\B(y_i)\cap\rho^{-1}(\red)\big)
		\dcup
		\big(N_\R(y_i)\cap\rho^{-1}(\blue)\big)\big|
		&>
		\big|N(y_i)\big|-\frac{p^2n}{200}-10\log n\nonumber\\
		&>
		\big|N(y_i)\big|-\frac{p^2n}{100}\,.\label{eq:comp1}
	\end{align}
	In other words, the preferred colour $\rho(v)$ of almost all 
	neighbours $v$ of $y_i$ mismatches the colour of the edge~$\{y_i,v\}$, i.e., 
	$\phi(\{y_i,v\})\neq\rho(v)$. 
	Next we show that both mismatching sets are large enough 
	to ensure quite a few edges crossing these sets. More precisely,  we will show 
	that the induced bipartite subgraph
\begin{equation}\label{eq:misE}
  \begin{split}
		&G_{\textrm{mis}}(y_i)=G\big[N_\B(y_i)\cap\rho^{-1}(\red),N_\R(y_i)\cap\rho^{-1}(\blue)\big]\\
		&\text{contains $p^2n/100$ vertices of degree at least  $p^2n/100$.}
  \end{split}
\end{equation}
	
	Note that the existence of any edge $\{u,v\}$ in the graph $G_{\textrm{mis}}(y_i)$
	allows us to connect~$y_i$ in colour~$\phi(\{u,v\})$ to
	the root of colour~$\phi(\{u,v\})$.  More precisely, if $u\in N_\B(y_i)\cap\rho^{-1}(\red)$ and 
	$v\in N_\R(y_i)\cap\rho^{-1}(\blue)$ and $\phi(\{u,v\})=\red$, then there exists a $\red$ 
	$y_i$-$r$-path using the $\red$ $u$-$r$-path guaranteed by~\eqref{eq:inv1} and the 
	$\red$ edges $\{y_i,v\}$ and $\{v,u\}$. This then would allow us to assign preferred 
	colour $\red$ to $y_i$. However, for a path as above we use $v$ for a $\red$ path, even though 
	$v$'s preferred colour is $\blue$ ($\rho(v)=\blue$). Such ``conflicts'' will be resolved in the second 
	stage and for that we need a more ``robust'' way to connect $y_i$ to the root of its 
	preferred colour. We prepare for that by proving~\eqref{eq:misE}.
        We also remark that the proof of~\eqref{eq:misE} 
	is the only place in the proof where it will be essential 
	that there is no monochromatic path between~$r$ and~$b$ and that $p\gg\big(\frac{\log n}{n}\big)^{1/2}$.
 
        \begin{proof}[Proof of~\eqref{eq:misE}]
          As it turns out, it suffices to establish a suitable lower
          bound on the cardinality of the two types of mismatching
          neighbourhoods of~$y_i$; namely, it is enough to prove that
	\begin{equation}\label{eq:lboundmiss}
		\big|N_\B(y_i)\cap\rho^{-1}(\red)\big|\geq \frac{1}{2}p^2n
		\qqand 
		\big|N_\R(y_i)\cap\rho^{-1}(\blue)\big|\geq\frac{1}{2}p^2n\,.
	\end{equation}
        Indeed, property~\ref{it:gnp:4} tells us
        that~\eqref{eq:lboundmiss} combined with~\eqref{eq:comp1}
        yields~\eqref{eq:misE}. 
	
	For the proof of~\eqref{eq:lboundmiss} we first observe that
	\begin{align}
		N_\B(y_i)\cap\rho^{-1}(\red)
		&=
		\big(N(y_i)\cap \rho^{-1}(\red)\big)\setminus \big(N_\R(y_i)\cap\rho^{-1}(\red)\big) \nonumber\\
		&\supseteq
		\big(N(y_i)\cap N(r)\cap \rho^{-1}(\red)\big)\setminus \big(N_\R(y_i)\cap\rho^{-1}(\red)\big)\,.\label{eq:Nyir}
	\end{align}
	We shall next consider the joint neighbourhood of~$y_i$
        and~$r$.  Note that no $v\in N_\B(r)$ can have preferred
        colour $\blue$.  In fact, if $\rho(v)=\blue$, then there
        exists a $\blue$ $v$-$b$-path in~$G$ (see~\eqref{eq:inv1}) and
        combined with $\phi(\{r,v\})=\blue$ this leads to a $\blue$
        path between~$r$ and $b$, which was excluded by the choice
        of~$r$ and~$b$. Moreover, every $\red$ neighbour~$v$ of~$r$
        outside $N_\R(r)\cap N_\B(b)\subseteq J$ (i.e.,
        every~$v\in N_\R(r)\setminus(N_\R(r)\cap N_\B(b))$) was
        assigned preferred colour $\red$
        in~\eqref{eq:prefN}. Therefore,
	\[
		N(r)\subseteq \rho^{-1}(\red)\cup J\cup Y_i\,,
	\]
        whence we deduce that
	\[
		N_\B(y_i)\cap\rho^{-1}(\red)
		\overset{\eqref{eq:Nyir}}{\supseteq}
		\big(N(y_i)\cap N(r)\big) \setminus \big(Y_{i+1}\cup J \cup (N_\R(y_i)\cap\rho^{-1}(\red))\big)\,.
	\]
	From~\eqref{eq:Ny}, the fact that~$y_i\not\in X$
        (see~\eqref{eq:X_def}), and the fact that~$y_i$ is
        not canonically connected in $\red$ (see~\eqref{eq:prefYs}), we infer that
	\begin{equation*}
		\big|N_\B(y_i)\cap\rho^{-1}(\red)\big|
		\geq
		\big|N(y_i)\cap N(r)\big|
			 -10\log n - \frac{p^2n}{200}-\frac{p^2n}{400}\,.
	\end{equation*}
	Therefore, the first inequality in~\eqref{eq:lboundmiss}
        follows from property~\ref{it:gnp:1} and $p^2n\gg \log n$. The
        second inequality in~\eqref{eq:lboundmiss} follows by the
        symmetric argument with colours exchanged.  As observed above,
        this establishes~\eqref{eq:misE} as well.
	\end{proof}

	Finally, we define the preferred colour of $y_i$ by
	\begin{equation}\label{eq:prefYns}
		\rho(y_i)=\begin{cases}
			\red, &\text{if $E(G_{\textrm{mis}}(y_i))\cap\phi^{-1}(\red)$ induces $\frac{p^2n}{200}$ 
				vertices of degree $\geq\frac{p^2n}{200}$}\\
		\blue, &\text{otherwise.}
		\end{cases}
	\end{equation}
	Recalling the discussion following~\eqref{eq:misE} we note that 
	also in this case we ensure property~\eqref{eq:inv1} for the vertex $y_i$. 
	Note that in view of property~\ref{it:gnp:4}, if $\rho(y_i)$ is blue, then $E(G_{\textrm{mis}}(y_i))\cap\phi^{-1}(\blue)$ induces $\frac{p^2n}{200}$ 
	vertices of degree $\geq\frac{p^2n}{200}$.
	
	This concludes the discussion of the first stage and we assigned preferred colours~$\rho(v)$
	to every vertex $v\in V\setminus \{r, b\}$. For that we considered 
	the vertices in $(N_\R(r)\cup N_\B(b))\setminus J$, in the joker set $J$, in the set $X$ connected 
	``robustly'' to the joker set, and in the remaining set $Y$ differently. Moreover, the 
	vertices in~$Y$ were treated differently depending on whether they are canonically connected or not.
	
	For later reference we note the following properties in addition to~\eqref{eq:inv1} for every vertex from the set
	$\big(J\setminus (N_\R(r)\cap N_\B(b))\big)\dcup X\dcup Y$.
	\begin{enumerate}[label=\alabel]
		\item\label{it:a} If $v\in J\setminus (N_\R(r)\cap
                  N_\B(b))$, then it follows from the definition~\eqref{eq:J_def}
                  of~$J$ that 
		\[
			\big|N_\R(v)\cap (N_\R(r)\setminus N_\B(b))\big|
			\geq 
			\frac{p^2n}{25}\,.
		\]
		\item\label{it:b} If $x\in X$, then it follows from~\eqref{eq:prefX}
		that 
		\[
			\big|N_{\rho(x)}(x)\cap J\big| \geq \frac{p^2n}{400}\,.
		\]
		\item\label{it:c} If $y_i\in Y$ is canonically connected in colour $\rho(y_i)$, 
			then it follows from~\eqref{eq:prefYs}
		that 
		\[
			\big|\big(N_{\rho(y_i)}(y_i)\setminus Y_i\big)\cap \rho^{-1}(\rho(y_i))\big| \geq \frac{p^2n}{400}\,.
		\]
              \item\label{it:d} If $y_i\in Y$ is not canonically
                connected in either colour, then by~\eqref{eq:prefYns}
                the bipartite subgraph of $G$ with edges of colour
                $\rho(y_i)$ induced across the two types of mismatched
                vertices in $N(y_i)\setminus Y_i$, which we denote by
                \[
                  G_{\rho(y_i)}\Big[\big(N_\B(y_i)\cap\rho^{-1}(\red)\big)\setminus
                  Y_i,\,
                  \big(N_\R(y_i)\cap\rho^{-1}(\blue)\big)\setminus
                  Y_i\Big]\,,
              \]
              contains at least $p^2n/200$ vertices of degree at least
              $p^2n/200$.
	\end{enumerate}
	
        \subsection*{Second stage: finalising the choices} We shall
        now assign \textit{final} colours to the vertices of~$G$ to
        establish~$\phi\ra\Pi_2$.  More precisely, we shall
        define a function $f\colon V\to\{\red,\blue\}$
        with~$f(r)=\red$ and~$f(b)=\blue$ so that
        \begin{equation}
          \label{eq:aim}
          G_\R[f^{-1}(\red)]\qand G_\B[f^{-1}(\blue)]
          \quad\text{are connected.}
        \end{equation}
        Since our process for defining~$f$ is somewhat lengthy, we
        first give a rough outline.  The assignment of the
        colours~$f(v)$ for~$v\in V$ will be achieved in two rounds.

        The function~$f$ will start as a partial function with
        domain~$\dom f$ close to half of~$V$.  At this stage, on most
        of~$\dom f$, we shall have~$f\equiv\rho$, but for about half
        of the joker vertices~$v$ we shall `switch' and pick as~$v$'s
        final colour the colour opposite to its preferred colour:
        $f(v)=\overline\rho(v)$, where~$\overline\rho(v)=\red$
        if~$\rho(v)=\blue$ and~$\overline\rho(v)=\blue$
        if~$\rho(v)=\red$.  At this point, we shall have that
        \begin{equation}
          \label{eq:rem1}
          G_\R[f^{-1}(\red)\setminus Y]\qand
          G_\B[f^{-1}(\blue)\setminus Y]
          \quad\text{are connected.}
        \end{equation}
        (The comment above is somewhat similar to Remark~\ref{rem:1}.)
        From this point in the proof onwards, we shall
        increase~$\dom f$ in smaller steps.  It will be convenient to
        say that, once~$f(v)$ has been defined for a vertex~$v$, the
        vertex~$v$ has been \textit{finalised}.  Also, we remark that,
        once we choose the value of~$f(v)$ for some~$v$, we shall not
        change it afterwards.

        What we discussed above corresponds to most of the first
        round.  However, still in the first round, we shall have to
        finalise some other vertices~$z\notin\dom f$,
        setting~$f(z)=\overline\rho(z)$ so that we can
        improve~\eqref{eq:rem1} by replacing~$Y$ by some substantially
        smaller subset~$Y'$ (in fact,~$|Y'|$~will roughly
        be~$|Y|/2$).  This final stage of the first round is
        encapsulated in Claim~\ref{claim:Z1} below.

        In the second round of our procedure defining~$f$, we pick the
        colour of the remaining vertices~$v\in V\setminus\dom f$.
        This process will be guided by the vertices in~$Y'$.  This
        concludes our outline of what comes next, and we proceed to
        define~$f$ precisely.

	Consider a random bipartition $Z_0\dcup Z_1=V\setminus\{r,b\}$
        where every vertex $v\in V\setminus\{r,b\}$ is included
        independently with probability~$1/2$ into $Z_0$ or $Z_1$.
        Since $p^2n\gg\log n$ we deduce from~\ref{it:a}--\ref{it:d}
        that with positive probability there exists a partition
        $Z_0\dcup Z_1=V\setminus\{r,b\}$ such that for every vertex
        in~$\big(J\setminus N_\R(r)\cap N_\B(b)\big)\dcup X\dcup Y$
        the following holds:
	\begin{enumerate}[label=\aplabel]
		\item\label{it:a'} If $v\in J\setminus (N_\R(r)\cap N_\B(b))$, then		
			$N_\R(v)\cap (N_\R(r)\setminus N_\B(b))\cap Z_0\neq \emptyset$.
		\item\label{it:b'} If $x\in X$, then $N_{\rho(x)}(x)\cap J\cap Z_\xi\neq\emptyset$
		for both~$\xi\in\{0,1\}$.
		\item\label{it:c'} If $y_i\in Y$ is canonically connected in colour $\rho(y_i)$, 
			then 
		\[
			\big(N_{\rho(y_i)}(y_i)\setminus Y_i\big)\cap \rho^{-1}(\rho(y_i))\cap Z_0\neq\emptyset\,.
		\]
              \item\label{it:d'} If $y_i\in Y$ is non-canonically
                connected, then there exists an edge
                $\{u,v\}\in E(G_{\rho(y_i)})$ such that
			\begin{align*}
				u&\in\big(N_{\overline{\rho}(y_i)}(y_i)\cap\rho^{-1}(\rho(y_i))\cap Z_0\big)\setminus Y_i
			\intertext{and}
				v&\in\big(N_{\rho(y_i)}(y_i)\cap\rho^{-1}(\overline{\rho}(y_i))\cap Z_1\big)\setminus Y_i\,,
			\end{align*}
			where, we recall, $\overline{\rho}(y_i)$ denotes the colour different from $\rho(y_i)$. 
	\end{enumerate}
	Note that we
	considered at most $n$ such sets 
	of size $\Omega(p^2n)$ in~\ref{it:a}--\ref{it:c} and $O(n\cdot p^2n)=O(n^2)$ stars of size 
	$\Omega(p^2n)$ in~\ref{it:d}. Consequently, the existence of a
        partition $Z_0\dcup Z_1=V\setminus\{r,b\}$  
	satisfying~\ref{it:a'}--\ref{it:d'} indeed follows from $p^2n\gg \log n$ and a standard application 
	of Chernoff's inequality.  We fix such a partition for the
        remainder of the proof.
	
	After this preparatory random splitting we start defining the final colours $f(v)$ for $v\in V$.
        We start with~$r$ and~$b$ in the obvious manner: 
	\[
		f(r)=\red
		\qqand 
		f(b)=\blue\,.
	\]
	Moreover, every $v\in Z_0$ will be assigned its preferred
        colour and every joker vertex in~$Z_1$ will be assigned the
        opposite of its preferred colour:
        \begin{equation}
          \label{eq:f_def_i}
          f(v)=
          \begin{cases}
            \rho(v),          &\text{if }v\in Z_0\\
            \overline\rho(v), &\text{if }v\in J\cap Z_1\,.
          \end{cases}
        \end{equation}
        Note that we now have~$\dom f=Z_0\cup J$.  We have thus
        committed ourselves in which of the two monochromatic
        subgraphs in~\eqref{eq:aim} the vertices in~$Z_0\cup J$ are.
        We mention that, owing to the definition of~$\rho$, our
        \textit{tendency} is to set~$f(v)=\rho(v)$ for the remaining
        vertices~$v\in Z_1\setminus J=V\setminus\dom f$.  However, if
        we do this blindly, assertion~\eqref{eq:aim} will not hold.
        In what follows, we shall ``switch'' the colour of
        \textit{some} vertices~$v\in Z_1\setminus J$ and we shall
        set~$f(v)=\overline\rho(v)$ (in the same way we did for the
        vertices in~$Z_1\cap J$).  Such switchings will basically be
        forced on us as we proceed to increase~$\dom f$ in our proof.

        Before we continue, we make the following remark, which is
        closely related to the discussion in Remark~\ref{rem:1}.
	
    \begin{remark}\label{rem:2}
      Suppose every vertex of~$Y$ is canonically connected in some
      colour.  Then properties~\ref{it:a'}--\ref{it:c'} and an
      inductive argument would show that~\eqref{eq:aim} holds for our
      current function~$f$.
      \hfill$\sqbullet$
    \end{remark}

    Remark~\ref{rem:2} above deals with the lucky case in which every
    vertex of~$Y$ is canonically connected in some colour.  In general,
    there will be vertices~$y$ in~$Y$ that are non-canonically connected.  
    Such vertices~$y$ will force us to
    set~$f(z)=\overline\rho(z)$ for some $z\in Z_1\setminus J$ also.
    This is made precise in the following claim.
	
    \begin{claim}\label{claim:Z1}
      There exists a subset $Z'_1\subseteq Z_1\setminus J$ for which
      the following holds.  If we set
      \begin{equation}\label{eq:finalZ1}
        f(z)=\overline\rho(z)
      \end{equation}
      for every $z\in Z'_1$, then~$\dom f=Z_0\cup J\cup Z_1'\cup\{r,b\}$
      and~\eqref{eq:aim} holds.
    \end{claim}
    \begin{proof}
      We first consider our current function~$f$
      with~$\dom f=Z_0\cup J$ and verify the following fact.

      \begin{fact}
        \label{fact:1}
        Assertion~\eqref{eq:rem1} holds for~$f$.
      \end{fact}
      \begin{proof}
        We consider the different types of vertices encountered in the first stage separately.
	First we recall that vertices $v\in(N_\R(r)\cup N_\B(b))\setminus J$ are directly connected 
	to their respective roots in colour $\rho(v)$. Consequently, all vertices  
	\[
		v\in Z_0\cap\big((N_\R(r)\cup N_\B(b))\setminus J\big)
	\]
	are in the same component in $G_{f(v)}=G_{\rho(v)}$ as the respective root. 
	
	Secondly, we consider the joker vertices. Note that nothing needs to be shown 
	for the vertices $v\in N_\R(r)\cap N_\B(b)$ as they are directly connected to both roots
	in the appropriate colour and, hence, for these vertices it does not matter which final 
	colour~$f(v)$ is assigned to them. Moreover, for every 
	joker vertex $v\in J\cap Z_0$ we have $f(v)=\rho(v)=\blue$ and since~$J\subseteq N_\B(b)$, 
	these vertices are also directly connected to~$b$ in~$G_{f(v)}$.
	For the remaining joker vertices 
	$v\in (J\setminus (N_\R(r)\cap N_\B(b)))\cap Z_1$ we appeal to~\ref{it:a'}.
	Owing to~\eqref{eq:f_def_i} the final colour~$f(v)$
        of~$v$ is $\red$ and, by~\ref{it:a'}, every such~$v$
	has at least one $\red$ neighbour $u$ in $Z_0\cap
        (N_\R(r)\setminus N_\B(b))\subset\dom f$.  Since 
	we have $f(u)=\rho(u)=\red$, the vertex~$v$ is also connected
        to~$r$ in~$G_\R[f^{-1}(\red)]$.  
	
	Next we move to the vertices $x$ in $X\cap Z_0$ and for those vertices we appeal to~\ref{it:b'}.
	If~$f(x)=\rho(x)=\red$, then~\ref{it:b'} applied with $\xi=1$ tells us that $x$ has at least one 
	$\red$ neighbour $v\in J\cap Z_1\subset\dom f$ (i.e.,
        there is~$v\in N_\R(x)\cap J\cap Z_1\subset\dom f$).  Since~$\rho(v)=\blue$ and, therefore, 
	$f(v)=\red$ (see~\eqref{eq:f_def_i}), we infer 
	from the discussion above that $x$ is connected by a $\red$ path to~$r$ in 
	$G_\R[f^{-1}(\red)]$.  If $f(x)=\rho(x)=\blue$, 
	then the same argument with~\ref{it:b'} applied with $\xi=0$ yields 
	that $x$ is connected by a $\blue$ path to~$b$ in~$G_\B[f^{-1}(\blue)]$.
      \end{proof}

      We shall now improve Fact~\ref{fact:1}: we shall prove
      that~\eqref{eq:aim} holds for~$f$, as long as we enlarge the
      domain of~$f$ suitably.  Roughly speaking, what we have to do is
      to `attach' the vertices in~$Y\cap Z_0$ to~$G_\R[f^{-1}(\red)]$
      or to~$G_\B[f^{-1}(\blue)]$, with edges (or paths) of the
      correct colour.  We shall proceed vertex by vertex following the
      order~$y_1,\dots,y_m$ (ignoring vertices outside~$Z_0$).  For
      certain vertices~$y_i\in Y\cap Z_0$, this will be a matter of
      realizing that a suitable edge is there; for other
      vertices~$y_i\in Y\cap Z_0$, we may have to finalise a
      vertex~$v\in Z_1\setminus J$: every time we do this, we add~$v$
      to~$Z'_1$ and~$Z'_1$ increases (we start with~$Z'_1=\emptyset$).
      Let us remark that, when we put a vertex~$v$ in~$Z'_1$ and
      finalise it, we shall set~$f(v)=\overline\rho(v)$.  At the end
      of this process, assertion~\eqref{eq:aim} will hold for our~$f$.
      We now go into the details of this process.

      We proceed inductively and use the fixed ordering of the
      vertices in~$Y$.  At first we have~$\dom f=Z_0\cup J$
      and~$Z'_1=\emptyset$.  Suppose now that~$1\leq i\leq m$,
      $y_i\in Y\cap Z_0$, and the vertices in some
      set~$Z'_1\subset Z_1\setminus J$ have been finalised
      with~$f(z')=\overline\rho(z')$ for every~$z'\in Z'_1$.  Suppose
      further that
      \begin{equation}
        \label{eq:aim_i}
        G_\R[f^{-1}(\red)\setminus Y_i]\qand
        G_\B[f^{-1}(\blue)\setminus Y_i]
        \quad\text{are connected.}
      \end{equation}
      We now finalise~$y_i$ analysing two cases.

      \textit{Case 1}.  If $y_i$ is canonically connected in
      colour~$\rho(y_i)$, then we proceed in a similar manner as for
      the vertices in~$X\cap Z_0$.  In fact, it follows
      from~\ref{it:c'} that in this case~$y_i$ has a neighbour
      $v\in N_{\rho(y_i)}(y_i)\setminus Y_i$ such that
      \[
        f(y_i)=\rho(y_i)=\rho(v)=f(v)\,,
      \] 
      where the first and last identities follow from the fact
      that~$y_i\in Z_0$ and~$v\in Z_0$.  Since
      $v\in(\dom f)\setminus Y_i$, the inductive
      assumption~\eqref{eq:aim_i} and the edge~$\{y_i,v\}$ of
      colour~$\rho(y_i)=f(y_i)$ tells us
      that~$G_{f(y_i)}[f^{-1}(f(y_i))\setminus Y_{i+1}]$ is
      connected, completing the induction step in this case. 
      
      \textit{Case 2}.  We now consider the case in which
      $y_i\in Y\cap Z_0$ is non-canonically connected.  
      In this case we may have to enlarge the set~$Z'_1$ by
      adding some vertex~$v$, but we will ensure the monochromatic
      connection for~$v$ as well.  By symmetry we may assume that the
      preferred colour of $y_i$ is $\red$ and, since $y_i\in Z_0$, we
      have
      \[
	\rho(y_i)=f(y_i)=\red.
      \]
      Let $\{u,v\}$ be the edge given by~\ref{it:d'} of colour $\rho(y_i)=\red$. In particular, 
      \[
	u\in(\rho^{-1}(\red)\cap Z_0)\setminus Y_i\,.
      \]
      Therefore, we already finalised~$u$ and~$f(u)=\red$.
      Furthermore, by the induction assumption~\eqref{eq:aim_i},
      we already know that $u$ is connected to~$r$ 
      by a $\red$ path in $G_\R[f^{-1}(\red)\setminus Y_i]$. 
      Furthermore, 
      \[
	v\in(\rho^{-1}(\blue)\cap Z_1)\setminus Y_i\,.
      \]
      In case $v$ has already been put into $Z'_1$ in this inductive
      process, then we already ``switched'' its colour and finalised
      it to be $\red$.  If not, then we add $v$ to $Z'_1$ at this
      point and finalise it with~$f(v)=\red$.  In any case we may use
      the $\red$ edges $\{u,v\}$ and $\{v,y_i\}$ to connect the
      vertices $v$ and $y_i$ to~$r$ by a $\red$ path in
      $G_\R[f^{-1}(\red)\setminus Y_{i+1}]$.  This concludes our
      induction step in this case and completes 
      the proof of Claim~\ref{claim:Z1}.
    \end{proof}

	It is left to finalise the colours of the vertices in
        $Z_1\setminus(J\cup Z'_1)$. Again we consider the vertices 
	separately, according to their membership in the sets $N_\R(r)\cup N_\B(b)$, $X$ or $Y$. This time
	we reverse the order in which we deal with the vertices and begin with the vertices in~$Y$. 
	
        We iterate over the vertices
        in~$Y\cap\big(Z_1\setminus(J\cup Z'_1)\big)$ in
        \textit{reverse order}: $y_m,\dots,y_1$.  In this process, we
        shall finalise the vertices~$y\notin\dom f$ that we encounter
        one by one.  For some~$y$, it may happen that some other
        vertex~$v\notin \dom f$ has to be finalised also.  When this
        does happen, we shall say that~$v$ has been \textit{pulled
          forward} and we shall always let~$f(v)=\overline\rho(v)$,
        that is, we shall switch the colour of~$v$.  We now describe
        this inductive process precisely.

        Let~$i\in[m]$ be the largest index such that~$y_i$ has not
        been finalised yet. We proceed as in the proof of
        Claim~\ref{claim:Z1}. If $y_i$ is canonically connected in
        colour~$\rho(y_i)$, then we set~$f(y_i)=\rho(y_i)$. Owing
        to~\ref{it:c'} there exists a neighbour in
        $v\in N_{\rho(y_i)}(y_i)\cap Z_0$ with preferred colour
        $\rho(v)=\rho(y_i)$. Since $v\in Z_0$, in fact, we already
        have~$f(v)=\rho(v)$ and, in view of Claim~\ref{claim:Z1}, the
        vertex~$v$ is connected to the root of the corresponding
        colour with an $f(v)$-coloured path.  Extending this path with
        the edge $\{v,y_i\}$ of colour $f(y_i)=f(v)$ to~$y_i$
        concludes this case.
	
	Next we consider the case in which $y_i$ is non-canonically
        connected.  In this case we also set
        $f(y_i)=\rho(y_i)$, but we shall make use of the edge
        $\{u,v\}$ of colour $\rho(y_i)$ guaranteed by~\ref{it:d'}.
        Since $u\in\rho^{-1}(\rho(y_i))\cap Z_0$, the colour~$f(u)$ of~$u$
        was chosen in the first round of the second stage already, and
        we have $f(u)=\rho(u)=\rho(y_i)=f(y_i)$.  Claim~\ref{claim:Z1}
        then tells us that there is a path from $u$ to the root of
        colour~$f(y_i)$ in $G_{f(y_i)}[f^{-1}(f(y_i))]$.  On the other
        hand, the vertex $v$ is contained in~$Z_1\setminus Y_i$ and
        $\rho(v)=\overline{\rho}(y_i)$.  We now proceed differently
        depending on whether or not~$v\in\dom f$.

	If $f(v)$ has not been set already, then we pull this vertex
        forward and finalise its colour opposite to its preferred
        colour, i.e., we treat the vertex~$v$ as the vertices
        $z\in Z'_1$ in~\eqref{eq:finalZ1}.  As a result we obtain
        $f(v)=f(y_i)$ and, since the edges $\{u,v\}$ and $\{v,y_i\}$
        are coloured $f(y_i)$, we ensure the invariant that $y_i$ and
        $v$ are connected to the root of colour $f(y_i)=f(v)$ in~$G_{f(y_i)}[f^{-1}(f(y_i))]$.
	
	If $f(v)$ has already been set before, then either
        (\textit{a})~$v\in(J\cap Z_1)\cup Z'_1$ and,
        by~\eqref{eq:f_def_i} and~\eqref{eq:finalZ1},
        the final colour of~$v$ was set
        opposite to its preferred colour, or else (\textit{b})~$v$ was
        pulled forward because of some other vertex $y_j$ with $j>i$.
        However, also in case~(\textit{b}), the colour of~$v$ was
        switched and we have
        $f(v)=\overline{\rho}(v)=\rho(y_i)=f(y_i)$.  Consequently, in
        both cases~(\textit{a}) and~(\textit{b}), we already
        established a connection of~$v$ to the root of colour~$f(v)$
        in $G_{f(v)}[f^{-1}(f(v))]$.  Extending this path with the
        edge $\{v,y_i\}$ of colour $f(v)=f(y_i)$ establishes the
        required connection for~$y_i$.  Here, we are using that
        $v\in Z_1\setminus Y_i$ being in~$(J\cap Z_1)\cup Z'_1$ or
        being pulled forward are the only reasons that could have led
        to the finalisation of~$v$.  This concludes the discussion of
        the vertices in~$Y$.
	
	Next we move to the vertices in~$X$.  Note that some of the
        vertices $x\in X\cap(Z_1\setminus(Z'_1\cup J))$ may have
        been pulled forward to attach some~$y\in Y$ that is 
        non-canonically connected.  However, such a
        vertex~$x$ was finalised and the desired connection to the
        root of colour~$f(x)$ was established on that occasion.

        For every vertex $x\in X\setminus\dom f$, we simply set
	\[
		f(x)=\rho(x)\,.
	\]
	By~\ref{it:b'} there exist vertices $u\in J\cap Z_0$ and
        $v\in J\cap Z_1$, both contained in $N_{f(x)}(x)$.
        Since all joker vertices were assigned preferred colour
        $\blue$ and~$u\in Z_0$, we have $f(u)=\rho(u)=\blue$.  On
        the other hand, since $v\in J\cap Z_1$, we infer
        from~\eqref{eq:f_def_i} that $f(v)=\red$.  Hence, no matter
        what $f(x)$ is, there exists a path from $x$ to the root of
        colour $f(x)$ in $G_{f(x)}[f^{-1}(f(x))]$.
	
	It is left to finalise the remaining vertices
        $v\in(N_\R(r)\cup N_\B(b))\cap (Z_1\setminus(Z'_1\cup J))$ that
        have not been pulled forward.  Obviously, setting $f(v)=\red$
        if $v\in N_\R(r)$ and $\blue$ otherwise connects~$v$ to the
        root in the appropriate colour. 
	
	Summarising, we finalised every vertex $v\in V$ in such a way
        that $v$ is connected to the root of colour~$f(v)$ in
        $G_{f(v)}[f^{-1}(f(x))]$ (i.e., assertion~\eqref{eq:aim}
        holds).  Consequently, the partition
	\[
		f^{-1}(\red)\dcup f^{-1}(\blue)=V
	\]
	shows that $\phi\ra\Pi_2$, which concludes the proof of
        Proposition~\ref{prop:ex}. 
\end{proof}

\section{Extension for more colours}\label{sec:counter}
In this section we show that Theorem~\ref{thm:main} does not extend in the expected way to 
more than two colours. For $r\geq 2$ and a graph $G=(V,E)$ we write $G\ra\Pi_r$ if for every 
$r$-colouring of $E$ there exist $r$ monochromatic trees $T_1,\dots,T_r\subseteq G$ such that 
\[
	V(T_1)\dcup \dots\dcup V(T_r)=V\,.
\] 
Since it is not hard to obtain a lower bound construction for the threshold $p=p(n)$ for $G(n,p)\ra\Pi_r$ as long as there are $r$ vertices with no joint neighbour, one may wonder whether
$p=p(n)=\left(\frac{r\ln n}{n}\right)^{1/r}$ is the sharp threshold for this property. Such a conjecture was indeed put forward by Bal and DeBiasio~\cite{BaDe}*{Conjecture~8.1}. However, it was noted by Ebsen, Mota, and Schnitzer~\cite{EMS} that for $r\geq 3$ the threshold is larger and we include their example below.

\begin{proposition}\label{prop:counter}
	For any integer $r\geq 3$ and $p=p(n)\ll \big(\frac{\ln n}{n}\big)^{\frac{1}{r+1}}$ a.a.s.\ $G\in G(n,p)$ 
	fails to satisfy $G\ra \Pi_r$.
	 
\end{proposition}
\begin{proof} For a simpler presentation we only prove the proposition for $r=3$, since the 
adjustments for $r>3$ are rather straightforward.
Suppose $p=p(n)\ll \big(\frac{\log n}{n}\big)^{1/4}$.
We show that a.a.s.\ $G=(V,E)\in G(n,p)$ admits a $3$-colouring of $E$ with colours $\red$, $\blue$, 
and $\green$ such that there is no partition $V(G)=V(T_1)\dcup V(T_2)\dcup V(T_3)$ with monochromatic 
trees $T_1$, $T_2$, $T_3\subseteq G$.
 
By our choice of~$p$ a.a.s.\ there are four vertices $r$, $b$, $g$, and $z$ that are independent in~$G$ and that have 
no common neighbour, i.e.,  
\[
	N(z)\cap N(r)\cap N(b)\cap N(g)=\emptyset\,.
\]
Below we write $N(r,g,b)$ for the joined 
neighbourhood~$N(r)\cap N(g)\cap N(b)$.

We now describe a colouring $\phi\colon E\to\{\red,\blue,\green\}$ with the desired property. 
The edges incident to $r$ are coloured $\red$, those incident to $b$ are coloured $\blue$, 
and those incident to $g$ are coloured $\green$. This choice ensures that we need at least three 
monochromatic trees to partition $V$ and below we will ensure that $z$ cannot be connected to any of these three trees. 

Next we colour the edges induced in 
\[
	X=N(r)\cup N(b)\cup N(g)\,. 
\]
in such a way that for every vertex $x\in X\setminus N(r,b,g)$, the edges incident to $x$ are coloured with at most 
two of the three colours and we fix one of the ``missing colours'' that do not appear on edges incident to $x$, 
which we denote by $\mc(x)$.
The following 
colourings have this property:

For every edge we list at least one \emph{allowed} colour and if 
an edge is assigned to more than one allowed colour, then one may pick arbitrarily one of the allowed colours
\begin{itemize}
	\item edges within $N(r)$ are allowed to be coloured $\red$, within $N(b)$ are allowed to be 
		coloured $\blue$, and within $N(g)$ are allowed to be coloured $\green$; 
	\item edges 
		between $N(r)\setminus N(b)$ and $N(b)\setminus \big(N(r)\cup N(g)\big)$ are coloured $\red$,
		between $N(b)\setminus N(g)$ and $N(g)\setminus  \big(N(b)\cup N(r)\big)$ are coloured $\blue$, and 
		between $N(g)\setminus N(r)$ and $N(r)\setminus  \big(N(g)\cup N(b)\big)$ are coloured $\green$.
\end{itemize}

Then we colour the edges incident with $z$. 
Edges $zx$ with $x\in X\setminus N(r,b,g)$ are coloured with colour $\mc(x)$.
Note that from the definition of $\mc(x)$, if $zx$ is coloured green, then there is no monochromatic green path between $g$ and $x$, and similar for the symmetric cases.

Let $Y$ be the set of vertices not considered so far, i.e., $Y=V(G)\setminus (X\cup \{r,b,g,z\})$.
It remains to colour the edges incident to $Y$.
We will prevent $z$ to be connected by a monochromatic path to $r$, $b$, or $g$ using vertices from~$Y$.
For that, we give colour blue to the edges $zy$ with $y\in Y$, while edges between $N(r,b,g)$ 
and $Y$ and within $Y$ are coloured~$\red$.
For the edges~$yx$ with $y\in Y$ and $x\in X\setminus N(r,b,g)$, the colours $\{\red,\green\}\setminus \{\mc(x)\}$ are allowed.
Since for every $x\in X\setminus N(r,b,g)$ and every $y\in Y$ the colours of the edges~$zx$ and $yx$ are different, and the only edge incident to $x$ that has colour $\mc(x)$ is $zx$, there is no monochromatic path from $x$ to $r$, $b$ or $g$ containing vertices from $Y$.
Moreover, one can check that for any colouring $\phi$ as described, it is impossible to connect $z$ by a monochromatic path with $r$, $b$, or $g$ and, hence $\phi$ has the desired property. 
\end{proof}

It would be interesting to determine the threshold for $G(n,p)\ra\Pi_r$ for $r\geq 3$ and to decide if the lower bound in Proposition~\ref{prop:counter} is optimal.
We remark that the construction given in Proposition~\ref{prop:counter} also works for covering (instead of partitioning) the vertices 
of~$G(n,p)$ with monochromatic trees.

\subsection*{Acknowledgement} The first author thanks Louis DeBiasio for introducing 
him to the problems considered in~\cite{BaDe}.

\BibSpec{personal}{%
  +{}{\PrintAuthors} {author}
  +{,}{ \textit{personal communication}} {transition}
  +{,}{ } {date}
  +{.}{ } {transition}
}

\end{document}